\documentclass[12pt,reqno]{amsart}
\usepackage[active]{srcltx}
\usepackage{verbatim}
\usepackage{epsfig,graphicx,color,mathrsfs}
\usepackage{graphicx}
\usepackage{amssymb}
\usepackage[english]{babel}
\usepackage[T1]{fontenc}
\usepackage[latin9]{inputenc}
\usepackage{amsmath}
\usepackage{amsthm}
\usepackage{amssymb}
\usepackage{xcolor}
\usepackage{geometry}
\makeatletter

\newtheorem{lemma}{Lemma}[section]

\newtheorem{proposition}[lemma]{Proposition}
\newtheorem{theorem}[lemma]{Theorem}

\theoremstyle{definition}

\newtheorem{remark}[lemma]{Remark}

\numberwithin{equation}{section}

\newcommand{\Bflat}{\ensuremath \mathcal{B}_{ft}}

\newcommand{\ib}{\ensuremath \displaystyle \int_{\Bflat^+}}
\newcommand{\ibm}{\ensuremath \displaystyle \int_{\Bflat^-}}
\thanks{$^*$corresponding author}

\title[Optimal second order boundary regularity]{Optimal second order boundary regularity for solutions to $p$-Laplace equations.}

\email{montoro@mat.unical.it., muglia@mat.unical.it, sciunzi@mat.unical.it}

\address{Department of Mathematics and Computer Science, UNICAL, Rende (CS), Italy}
\author[L.\ Montoro]{Luigi Montoro}
\author[L.\ Muglia]{Luigi Muglia}
\author[B.\ Sciunzi]{Berardino Sciunzi $^*$}

\begin{document}

\begin{abstract}
Solutions to $p$-Laplace equations are not, in general, of class $C^2$. The study of  Sobolev regularity of the second derivatives is, therefore,  a crucial issue. 
An important contribution by Cianchi and Maz'ya shows that, if the source term is in $L^2$, then the field $|\nabla u|^{p-2}\nabla u$ is in $W^{1,2}$. The $L^2$-regularity of the source term is also a necessary condition. Here, under suitable  assumptions, we obtain sharp second order estimates, thus proving  the optimal regularity of the vector field $|\nabla u|^{p-2}\nabla u$, up to the boundary.  
\end{abstract}

\maketitle

\section{Introduction}
We deal with the study of the regularity of weak solutions to
\begin{equation}\label{EP}
	-\Delta_p u=f(x)
\end{equation}
in a bounded smooth domain $\Omega\subset \mathbb{R}^n$, $p>1$. In all the paper we will assume that
\[
f\in W^{1,1}(\Omega)\cap L^q(\Omega),\qquad q>N\,.
\]
We shall take into account the classical Dirichlet or Neumann boundary condition. 
Namely $u\in W_0^{1,p}(\Omega)$ and
\begin{equation}\label{wsD}
	\int_\Omega|\nabla u|^{p-2}\langle \nabla u,\nabla \varphi\rangle dx=\int_\Omega f\varphi dx, \qquad \varphi\in C_c^\infty(\Omega)
\end{equation}
for the Dirichlet condition, or $u\in W^{1,p}(\Omega)$ and
\begin{equation}\label{wsN}
	\int_\Omega|\nabla u|^{p-2}\langle \nabla u,\nabla \varphi\rangle dx=\int_\Omega f\varphi dx, \qquad \varphi\in C^\infty(\Omega)
\end{equation}
for the Neumann condition with  $|\nabla u|^{p-2}\partial_\nu u=0$ on $\partial \Omega$, where $\nu$ is the outward normal vector. Note that, by \cite{DiB, Lieb, LiebQ}, see also \cite{DuMi1, DuMi2, Ev, LaUr, Tolk, Uh}, we have that
$u\in C^{1,\alpha}(\bar{\Omega})$. As it is well known, for the Neumann case, we implicitly assume that the source term $f$ has zero mean.\\

\noindent We consider here the issue regarding the study of the summability of the second derivatives, up to the boundary. The Calder\'on-Zygmund theory is  not completely well understood in the literature but very deep results have been achieved thanks to the works \cite{AvKuMi, CiMa1, CiMa3, CiMa2, Ce,  CoMi, DS, DeFMi, DeTh,  DuMi1, DuMi2, KuMi, GuMo, Lou, Mi1, Mi2, Si}. 

 A celebrated   result in the theory of elliptic equations states that, if 
$$
\Delta u=f
$$
with right term $f\in L^2(\Omega)$,  then $u\in W^{2,2}(\Omega)$.  
In \cite{CiMa} the authors prove the natural analogue to this classical result showing that 
\begin{equation}\label{CM}
|\nabla u|^{p-2}\nabla u\in W^{1,2}(\Omega) \mbox{ iff } f\in L^2(\Omega),
\end{equation}
under minimal assumptions on the regularity of the domain, when Dirichlet or Neumann boundary condition is imposed.

\

\noindent Here, under suitable stronger assumptions, we obtain sharp second order estimates, thus proving  the optimal regularity of the vector field $|\nabla u|^{p-2}\nabla u$, up to the boundary.

\begin{theorem}\label{Reg}
	Let $u\in C^{1,\alpha}(\bar{\Omega})$ be a weak solution of \eqref{EP}  with respect to \textit{Dirichlet condition} or \textit{Neumann condition}. Assume that $\Omega$ is of class $C^3$
	and
	\[
	f\in W^{1,1}(\Omega)\cap L^q(\Omega), \qquad q>N.
	\]
	Setting $V_\alpha:=|\nabla u|^{\alpha-1}\nabla u$. Then 	\begin{center}
		$V_\alpha\in W^{1,2}(\Omega)$ for $\displaystyle \alpha>\frac{p-1}{2}$.
	\end{center}
	As a corollary, for $p<3$, we deduce that $u\in W^{2,2}(\Omega)$.
	
\end{theorem}

\

\begin{remark}
	Note that 	by Theorem \ref{Reg} with $\alpha=p-1$ we recover
	$$
	|\nabla u|^{p-2}\nabla u\in W^{1,2}(\Omega).
	$$
	This is the result of Cianchi and Maz'ya  that was proved in \cite{CiMa} under the optimal assumption $f\in L^2(\Omega)$. 
\end{remark}
This kind of results are actually based on the idea of considering the vector field as a single weighted element. By the way in many applications it is better to know a pure Sobolev regularity of the second derivatives, as for the case $p<3$ in Theorem~\ref{Reg}. When $f$ does not vanish we have the following:
\begin{theorem}\label{Reg1}
	Let $u\in C^{1,\alpha}(\bar{\Omega})$ be a weak solution of \eqref{EP}  under  \textit{Dirichlet } or \textit{Neumann } condition. Assume that $\Omega$ is of class $C^3$
	and
	\[
	f\in W^{1,1}(\Omega)\cap L^q(\Omega), \,\, q>N\quad with \quad  f\leq -\tau <0\,\, \text{or}\,\, f\geq \tau >0 \quad a.e. \,\,in \quad  \Omega\,.
	\]
 Then 
\[u\in W^{2,q}(\Omega),\,\, 1\leq q<{(p-1)}/{(p-2)},\quad\text{if}\,\, p \geq3\,.\]
\end{theorem}

\

\begin{remark} \textbf{Optimality of the result}:\\
	\noindent 
	Note that,  our regularity results are sharp, already in the case of interior estimates. Regarding Theorem \ref{Reg}, this can be deduced by the example involving   $u\,:=\, x_1^{p'}/p'$ with $x_1\geq 0$,  that solves $\Delta_p \,u=1$. 
	The same example shows the optimality of the result of  Theorem \ref{Reg1} ($p\geq 3$)
since  $u\in W^{2,q}(\Omega)$ only for $q<(p-1)/(p-2)$. Finally let us stress the fact that, under our general assumptions, our results are new also restricting the attention to the case of  interior estimates.
\end{remark}

Before starting our proofs, for the reader's convenience, we anticipate a key idea in our technique. 
It is well know that flattening arguments may be of use in the study of the regularity theory. Unfortunately the transformed equation, in general, is not in divergence form. We make use of a fine argument based on the Fermi coordinates to deal with the difficulty that arises and to manage many boundary terms. Actually the use of Fermi coordinates is a key ingredient in our proofs that can not be carried out  with standard flattening arguments. 
Once this procedure is well understood, we deduce our regularity results exploiting, also, some ideas from \cite{DS} for the study of local interior estimates. The extension of such sharp estimates up to the boundary is completely  not trivial and it has been an open problem till now.

\section{Flattening via Fermi coordinates}\label{sectionfermi}
In this section we introduce some notation and preliminary facts about the flattening operator.  Let us start with an assumption on the domain that we will recall in all the statements of our results\\
\begin{center}
	{\bf ($H_\Omega$)}  $\Omega$ is a  bounded $C^3$-smooth domain.
\end{center}\ \\
Then,  for any $\bar{x}\in \partial \Omega$, there exists $r>0$ and a ball $B_r(\bar{x})$ such that $\partial \Omega\cap B_r(\bar{x})$ can be represented  by a $C^3$-mapping $\gamma:D\subset \mathbb{R}^{N-1}\to\mathbb{R}^N$ with $\bar{D}\subset\mathbb{R}^N$ compact (and locally the domain is an epigraph or an hypograph).  For any $\varepsilon>0$ set
$$
B^+_{\varepsilon}:=\{x\in \Omega: \ d(x,\partial\Omega)<\varepsilon\}\cap B_r(\bar{x}).
$$
 Then there exists (see e.g. \cite{Foote}) $\varepsilon_0>0$ such that $B^+_{\varepsilon_0}$ has the unique nearest point property,  i.e.  for any $y\in B^+_{\varepsilon_0}$ there exists a unique $x\in \partial \Omega$ such that $d(x,y)=d(y,\partial \Omega)$.  

Then let us define the \emph{flattening operator} $\Phi:\mathbb{R}^n\to \mathbb{R}^n$ as follows: for all $x\in B^+_{\varepsilon_0}$, let $\bar{x}=\gamma(y_1,\ldots, y_{N-1})\in \partial \Omega$ and $^i_{\bar{x}}\eta=\ ^i\eta$ be the inward normal vector in $\bar{x}$. Then 
\begin{equation}\label{ft}
	x=\Phi(y):=\gamma(y_1,\ldots, y_{N-1})+{y_N} ^i\eta,
\end{equation}	
i.e. $x=(x_1\ldots,x_N)$ is such that
$$
x_j:=\gamma_j(y_1,\ldots,y_{N-1})+y_N \ ^i_x\eta_j(y_1,\ldots,y_{N-1}), \qquad \forall j=1,\ldots, N.
$$

Note that in the new local-coordinate the hyperplane $\{y_N=0\}$ represents (locally) the boundary points of $\Omega$.\\
In the sequel we will indicate by $\Bflat^+$ the set of the $y$-space such that 
\begin{equation}\label{Bflat}
	B_{\varepsilon_0}^+=\Phi(\Bflat^+).
\end{equation}
For our scope, it will be necessary to define $\Bflat^-$ as the reflection of $\Bflat^+$ with respect to the hyperplane $\{y_N=0\}$ and 
$$\Bflat:=\Bflat^+\cup \Bflat^-$$ 
respectively.\\
Let us consider problems \eqref{wsD} and \eqref{wsN} and let us see how they are transformed by   the  flattening Fermi operator \eqref{ft}.\\

For a given $\varphi\in C_c^\infty(B_{\varepsilon_0}^+)$ ($\varphi\in C^\infty(B_{\varepsilon_0}^+)$ for the Neumann problem,  respectively),  $(\varphi\circ\Phi)$ lies in $C_c^\infty(\Bflat^+)$ ($\varphi\in C^\infty(\Bflat^+)$ for the Neumann problem, respectively), hence $\eqref{wsD}$  and $\eqref{wsN}$ can be rewritten as
\begin{eqnarray}\label{tf}
	\nonumber&&\int_{\Bflat^+}
	|\nabla_x u(\Phi(y))|^{p-2}\langle\nabla_x u(\Phi(y)),\nabla_x\varphi(\Phi(y))\rangle |\hbox{det} J_\Phi(y)|dy\\
	&=&\int_{\Bflat^+}f(\Phi(y))\varphi(\Phi(y))|\hbox{det} J_\Phi(y)|dy.
\end{eqnarray}\ \\
In what follows, let us set
$$w(y):=u(\Phi(y)), \  g(y):=f(\Phi(y)),  \ \phi:=\varphi(\Phi(y)),  \  \delta(y):=|{det} J_\Phi(y)|.
$$
It is easy to check that
$$
\nabla_y w(y)= J_\Phi(y)^T\nabla_x u(\Phi(y))\,,
\quad
\nabla_y \phi(y)
= J_\Phi(y)^T\nabla_x \varphi(\Phi(y)).
$$
Then, defining $A(y):=[J_\Phi(y)^T]^{-1}$ and $K(y):=A(y)^T A(y)$
our equation \eqref{tf} can be written as
	\begin{eqnarray}\label{t1}
			\int_{\Bflat^+}
			|A(y)\nabla w(y)|^{p-2}\langle K(y)\nabla w(y),\nabla \phi(y)\rangle \delta(y)
			dy
			=\int_{\Bflat^+}g(y)\phi(y)\delta(y)dy,
		\end{eqnarray}
		for any $\phi \in C_c^\infty(\Bflat^+)$ (respectively for any  $\phi \in C^\infty(\Bflat^+)$ for the Neumann problem).
		\begin{remark}\label{rem2}
			Note that, by the regularity of $\gamma$, there exists $C_A, c_A>0$ such that, for any $v\in \mathbb{R}^n$
			$$
			c_A|v|\leq |Av|\leq C_A|v|
			$$
			In a similar way let us consider $C_K, c_K$ two positive number such that $$c_K|v|\leq |Kv|\leq C_K|v|\,.$$
		\end{remark}
		Since for any invertible matrix $C$ the equality $(C^T)^{-1}=(C^{-1})^T$ holds, then
			\begin{equation}
				K=K(y):=A^T A=((J_\Phi^T)^{-1})^T(J_\Phi^T)^{-1}=(J_\Phi^TJ_\Phi)^{-1}
			\end{equation}
			therefore we can write $K$ as 
			\begin{equation}\label{kappagen}
				K=\left[\left (
				\begin{tabular}{lll}
					$\cdots$ & $\partial_{y_1}(\gamma+y_N \ ^i\eta)$ & $\cdots$ \\
					$\cdots$ & $\partial_{y_2}(\gamma+y_N \ ^i\eta)$ & $\cdots$ \\
					$\cdots$ & $\cdots$ & $\cdots$\\
					$\cdots$ & $\quad ^i\eta$ & $\cdots$ \\
				\end{tabular}
				\right ) \left (
				\begin{tabular}{lll}
					$(\nabla \gamma_1+y_N\nabla (^i\eta_1))$ & $|$ & $^i\eta_1$\\
					$(\nabla \gamma_2+y_N\nabla (^i\eta_2))$ & $|$ & $^i\eta_2$\\
					$\vdots$ & $|$ & $\vdots$\\
					$(\nabla \gamma_N+y_N\nabla (^i\eta_N))$ & $|$ & $^i\eta_N$\\
				\end{tabular}
				\right )\right]^{-1}.
			\end{equation}
			Moreover $K|_{\{y_N=0\}}$ has the form
			\begin{equation}
				K|_{\{y_N=0\}}=\left[\left (
				\begin{tabular}{lll}
					$\cdots$ & $\partial_{y_1}\gamma$ & $\cdots$ \\
					$\cdots$ & $\partial_{y_2}\gamma$ & $\cdots$ \\
					$\cdots$ & $\cdots$ & $\cdots$\\
					$\cdots$ & $\quad ^i\eta$ & $\cdots$ \\
				\end{tabular}
				\right ) \left (
				\begin{tabular}{lll}
					$\nabla \gamma_1$ & $|$ & $^i\eta_1$\\
					$\nabla \gamma_2$ & $|$ & $^i\eta_2$\\
					$\vdots$ & $|$ & $\vdots$\\
					$\nabla \gamma_N$ & $|$ & $^i\eta_N$\\
				\end{tabular}
				\right )\right]^{-1}.
			\end{equation}
			Taking into account the orthogonality between the normal and the tangent hyperplane, $K|_{\{y_N=0\}}$ is
			\begin{equation}\label{kappaBordo}
				K|_{\{y_N=0\}}=\left(
				\begin{tabular}{llll}
					$ \ $ & $ \ $ & $ \ $ & $0$ \\
					$ \ $ & $C(D\gamma)$ & $ \ $ & $0$ \\
					$ \ $ & $ \ $ & $ \ $ & $0$\\
					$0$ & $0$ & $\cdots$ & $1$ \\
				\end{tabular}
				\right)^{-1}.
			\end{equation}

		\begin{remark}\label{remimp}
			Consider  problem \eqref{EP} under the  Neumann condition. Here $^i\eta=-\nu$. If $x\in\partial \Omega\cap B_r(\bar{x})$ is such that $|\nabla u(x)|\neq 0$, calling $y\in\{y_N=0\}$ the point for which $x=\Phi(y)$, it results
			$$
			\partial_\nu u(x)=\langle \nabla u(x),\nu\rangle=-\langle A(y)\nabla w(y), A(y)e_N\rangle=-\langle K(y)\nabla w(y), e_N\rangle=-w_{y_N}(y)
			$$
			by \eqref{kappaBordo},  therefore, $0=w_{y_N}(y):=w_N(y)$.\\
			On the contrary, considering the case of the  Dirichlet condition, we only know that on the boundary $\{y_N=0\}$, $\nabla{w}=(\bar{0}, w_N)$ and $w_N=\langle K(y)\nabla w, e_N\rangle$.
		\end{remark}

		\section{The transformed problem}
		
		Here we study the equation obtained by the flattening argument via Fermi coordinates.
		It is well known that such a procedure does not provide, in general, equations in divergence form. By the way, we show the existence of a nice weak formulation. We start with the following:
			\begin{remark}\label{remregcianchi}
				Note that, under our assumptions, we have that 
				\[
				u\in C^{1,\alpha}(\overline{\Omega})
				\]
				since the source term is bounded. 
			As recalled in the introduction, in our proofs we make also use of the results
			of Cianchi and Maz'ya. In particular 
			$$
			|\nabla u|^{p-2}\nabla u\in W^{1,2}(\Omega).
			$$
			under the necessary and sufficient condition $f\in L^2(\Omega)$, see 
			  Theorem 2.4 in \cite{CiMa}. Actually we shall also exploit the fact that 
			  \begin{equation}\label{W22}
			   u\in W^{2,2}(\hat\Omega)
			  \end{equation}
			  for any domain $\hat\Omega\subseteq \Omega$ such that 
			  \begin{equation}\label{apprcia}
			   |\nabla u|\geq\theta>0 \quad in \quad \hat\Omega.
			  \end{equation}
			  This is a consequence of  Theorem 2.4 in \cite{CiMa}. Actually, to deduce this fact, we need to observe that the approximating argument of \cite{CiMa}, under our assumptions, produces a sequence of functions uniformly bounded in 	$ C^{1,\alpha}(\overline{\Omega})$ (see \cite{LiebQ}) so that the condition in \eqref{apprcia} is preserved for the approximating net, redefining the constant. 
		\end{remark}
	From now on, in our computations, we shall always use the regularity information in Remark \ref{remregcianchi} without mentioning it each time. Let us now define an even extension for the operator $A(\cdot)$ (and as a rule for $K$) and for the mapping $\delta(\cdot)$ on $\Bflat$ as follows
		$$
		\tilde{A}(y)=\begin{cases} A(y), &y_N \geq 0\\
			A(y_1, y_2, \dots, y_{N-1}, -y_N) &y_N<0;
		\end{cases}
		$$
		
		$$
		\tilde{\delta}(y)=\begin{cases} \delta(y), &y_N \geq 0,\,\\
			\delta(y_1, y_2, \dots, y_{N-1}, -y_N) &y_N<0.\,\\
		\end{cases}
		$$
\begin{remark}\label{grad-}
In the following, for a given vector 
$\mathcal{V}=(\mathcal{V}_1,\ldots,\mathcal{V}_N)$, we will set $$\mathcal{V}^-:=(\mathcal{V}_1,\ldots, -\mathcal{V}_N).$$ In particular we set 
\[\nabla^-u:=(\nabla u)^-.\]
\end{remark}
For a given differentiable function $h(x)$ we will indicate by $$Z_h:=\{x: \nabla h=0\}.$$
Let us now introduce the extension of our map $w$. Let $u\in C^{1,\alpha}(\bar{\Omega})$ be a weak solution to \eqref{EP} under the  Dirichlet condition and $w$ the solution to \eqref{t1}. Let us define $\tilde{w}\in C^{1,\alpha}(\Bflat)$ by odd extension as
$$
\tilde{w}(y)=\begin{cases} 
	w(y), & y_N \geq 0\\
	-w(y_1, y_2, \dots, y_{N-1}, -y_N) &y_N<0.
\end{cases}
$$
In a similar way, let $u\in C^{1,\alpha}(\bar{\Omega})$ be a weak solution to \eqref{EP} under the Neumann condition and $w$ the solution to \eqref{t1}. Let us define 
$\tilde{z}\in C^{1,\alpha}(\Bflat)$ defined by even extension as
$$
\tilde{z}(y)=\begin{cases} 
	w(y), & y_N \geq 0\\
	w(y_1, y_2, \dots, y_{N-1}, -y_N) &y_N<0.
\end{cases}
$$

\begin{proposition}\label{1risul} Let $\Omega\subset \mathbb{R}^n$ satisfying $(H_\Omega)$,  $f\in W^{1,1}(\Omega)\cap L^q(\Omega)$ with $q>N$, $\Phi$ the flattening operator defined in \eqref{ft} and $\Bflat^+$ and $\Bflat^-$ as above.  Let $u\in C^{1,\alpha}(\bar{\Omega})$ be a weak solution to \eqref{EP} under Dirichlet or Neummann conditions. Then $\tilde{w}(y)$ and $\tilde{z}(y)$ fulfil 
 \begin{eqnarray}\label{Neqfin}\\\nonumber
					\int_{\Bflat^-}
					|\tilde{A}\nabla^- v|^{p-2}\langle (\tilde{K}\nabla^- v)^-,\nabla \psi\rangle \tilde{\delta}dx
					+\int_{\Bflat^+}
					|\tilde{A}\nabla v|^{p-2}\langle \tilde{K}\nabla v,\nabla \psi\rangle \tilde{\delta}
					dx=\int_{\Bflat}\tilde{g}\psi \tilde{\delta}dx,
 \end{eqnarray}
				for all $\psi\in C_c^\infty(\Bflat)$.
\end{proposition}
			
			\begin{proof}
				Let us first consider the case of the  \emph{Dirichlet condition}.
				%
				Let us define the odd extension of  $g$ with respect to $y_N$ as follows
%
$$
				\tilde{g}(y)=\begin{cases} g(y) &y_N \geq 0\,\\
					-g(y_1, y_2, \dots, y_{N-1}, -y_N) &y_N<0\,\\
				\end{cases}
				$$
and let $\tilde w$ be defined as above.
					Let us consider for $y=(\bar{y},y_N)$ the new coordinate $x=(\bar{x},x_N)$ where $\bar{x}=\bar{y}$ and $x_N=-y_N$; equation  \eqref{t1}   becomes
					\begin{eqnarray}\label{t2}
						\int_{\Bflat^-}
						|\tilde{A}(x)\nabla^- \tilde{w}(x)|^{p-2}\langle \tilde{K}(x)\nabla^- \tilde{w}(x),\nabla^- \tilde{\phi}(x)\rangle\tilde{\delta}(x)
						dx
						=\int_{\Bflat^-}\tilde{g}(x)\tilde{\phi}(x)\tilde{\delta}(x)dx,
					\end{eqnarray}
					where 
					$\tilde{\phi}(x):=\phi(\bar{x},-x_N)\in C_c^\infty(\Bflat^-)$. Then taking into account the regularity results in \cite{CiMa},  \eqref{t1} and \eqref{t2} means that $\tilde{w}$ satisfies
					\begin{eqnarray}\label{NP1}
						-div(\tilde{\delta} |\tilde{A}\nabla \tilde{w}|^{p-2}\tilde{K}\nabla \tilde{w})=\tilde{g} \ \tilde{\delta} \quad \mbox{ a.e. on } \Bflat^+\ \\
						\nonumber  -div(\tilde{\delta}|\tilde{A}\nabla^- \tilde{w}|^{p-2}(\tilde{K}\nabla^- \tilde{w})^-)=\tilde{g} \ \tilde{\delta} \quad \mbox{ a.e. on } \Bflat^-,
					\end{eqnarray}
					respectively.
					Let $$\psi\in C_c^\infty(\Bflat\setminus Z_{\tilde{w}});$$ multiplying \eqref{NP1} for $\psi$ and integrating on the respective domain we get
					\begin{eqnarray}\label{Neq}
						\nonumber	&&\int_{\Bflat^-\cup \Bflat^+}\tilde{g}\tilde{\delta} \ \psi dx=\int_{\Bflat^-}
						|\tilde{A}\nabla^- \tilde{w}|^{p-2}\langle \tilde{K}\nabla^- \tilde{w},\nabla^- \psi\rangle \tilde{\delta} dx \\\nonumber&&-\int_S 
						|\tilde{A}\nabla^- \tilde{w}|^{p-2}\langle(\tilde{K}\nabla^- \tilde{w})^-,e_N\rangle \tilde{\delta}\psi \ ds+\int_{\Bflat^+}
						|\tilde{A}\nabla \tilde{w}|^{p-2}\langle \tilde{K}\nabla \tilde{w},\nabla \psi\rangle
						\tilde{\delta}dx\\&&- \int_S 
						|\tilde{A}\nabla \tilde{w}|^{p-2}\langle \tilde{K}\nabla \tilde{w},-e_N\rangle\tilde{\delta}\psi \ 
						ds,
					\end{eqnarray}
					where $S:=supp(\psi)\cap\{x_N=0\}$.
						Taking into account the Dirichlet condition and Remark \ref{remimp}, $\nabla \tilde{w}(\bar{x},0)=(\bar{0}, \tilde{w}_N(\bar{x},0))$ and $\nabla^- \tilde{w}(\bar{x},0))=(\bar{0}, -\tilde{w}_N(\bar{x},0))$;
						therefore 
						\begin{eqnarray}\label{bordodopo}
							\int_S |\tilde{A}\nabla^- \tilde{w}|^{p-2}\langle (\tilde{K}\nabla^- \tilde{w})^-,e_N\rangle \psi \tilde{\delta} ds&=&\int_S |\tilde{A}\nabla \tilde{w}|^{p-2}\tilde{w}_N \psi \tilde{\delta} ds,\\
							\nonumber \int_S |\tilde{A}\nabla \tilde{w}|^{p-2}\langle \tilde{K}\nabla \tilde{w},-e_N\rangle\psi \tilde{\delta}  ds
							&=&-\int_S |\tilde{A}\nabla \tilde{w}|^{p-2}\tilde{w}_N \psi \tilde{\delta} ds.
						\end{eqnarray}
						Hence \eqref{Neq} states that $\tilde{w}$ satisfies
						\begin{eqnarray}\label{Neqfinu}
							\int_{\Bflat^-\cup \Bflat^+}\tilde{g}\tilde{\delta}\psi dx=\int_{\Bflat^-}
							|\tilde{A}\nabla^- \tilde{w}|^{p-2}\langle \tilde{K}\nabla^- \tilde{w},\nabla^- \psi\rangle \tilde{\delta}dx\\
							\nonumber	+\int_{\Bflat^+}
							|\tilde{A}\nabla \tilde{w}|^{p-2}\langle \tilde{K}\nabla \tilde{w},\nabla \psi\rangle
							\tilde{\delta}dx,
						\end{eqnarray}
						for all $\psi\in C_c^\infty(\Bflat\setminus Z_{\tilde{w}})$.\\
Let us now consider the case of the \emph{Neumann condition}. Let us define the even extensions with respect to $y_N$ of  $g$ as follows 
						$$
						\tilde{g}(y)=\begin{cases} g(y) &y_N \geq 0\,\\
							g(y_1, y_2, \dots, y_{N-1}, -y_N)  &y_N<0,\,\\
						\end{cases}
						$$
and let $\tilde{z}$ be defined as above.			Arguing as for the Dirichlet condition, for $\psi\in C_c^\infty(\Bflat\setminus Z_{\tilde{w}})$, from \eqref{NP1} we get \eqref{Neq}:
						\begin{eqnarray*}
							\nonumber	&&\int_{\Bflat^-\cup \Bflat^+}\tilde{g}\tilde{\delta} \ \psi dx=\int_{\Bflat^-}
							|\tilde{A}\nabla^- \tilde{z}|^{p-2}\langle \tilde{K}\nabla^- \tilde{z},\nabla^- \psi\rangle \tilde{\delta} dx \\
							&&-\int_S
							|\tilde{A}\nabla^- \tilde{z}|^{p-2}\langle (\tilde{K}\nabla^- \tilde{z})^-,e_N\rangle \psi \tilde{\delta}ds+\int_{\Bflat^+}
							|\tilde{A}\nabla \tilde{z}|^{p-2}\langle \tilde{K}\nabla \tilde{z},\nabla \psi\rangle
							\tilde{\delta}dx\\
							&&- \int_S
							|\tilde{A}\nabla \tilde{z}|^{p-2}\langle \tilde{K}\nabla \tilde{z},-e_N\rangle\psi \tilde{\delta}
							ds,
						\end{eqnarray*}
						where $S:=supp(\psi)\cap\{x_N=0\}$. Note that, in this case, by  Remark \ref{remimp},
						\begin{eqnarray}\label{bordodopo2}
							\int_S \psi \tilde{\delta}|\tilde{A}\nabla \tilde{z}|^{p-2}\langle \tilde{K}\nabla \tilde{z},-e_N\rangle \ ds
							=0.				\end{eqnarray}
Furthermore by  Remark \ref{remimp} and Remark \ref{grad-} we obtain\begin{equation*}
\int_S \psi						|\tilde{A}\nabla^- \tilde{z}|^{p-2}\langle (\tilde{K}\nabla^- \tilde{z})^-,e_N\rangle \tilde{\delta}ds=0.
\end{equation*}
						To conclude, let us prove that \eqref{Neqfinu} holds for all $\psi \in C_c^\infty(\Bflat)$. For a given $\eta\in C_c^\infty(\Bflat)$ and defining 
						$$\displaystyle H_\epsilon(t):=\chi_{(2\epsilon,+\infty)}+\left(\frac{t}{\epsilon}- 1\right) \chi_{(\epsilon, 2\epsilon)},$$ let us consider 
						\[
							\psi:=\eta H_{\epsilon}(|\nabla w|)\in W_c^{1,2}(\Bflat)\,.
						\]
					 Note that $supp(\psi)\subset \Bflat\setminus Z_{\tilde{w}}$ and,
						by a standard density arguments,  we can  use it as test functions in \eqref{Neqfinu}. Then
						\begin{eqnarray*}
							\int_{\Bflat^-\cup \Bflat^+}\tilde{g}\eta H_{\epsilon}(|\nabla \tilde{w}|)\tilde{\delta}dx=\int_{\Bflat^-}
							|\tilde{A}\nabla^- \tilde{w}|^{p-2}\langle \tilde{K}\nabla^- \tilde{w},\nabla^- \eta\rangle H_{\epsilon}(|\nabla \tilde{w}|) \tilde{\delta}dx\\
							\int_{\Bflat^-} |\tilde{A}\nabla^- \tilde{w}|^{p-2}\langle \tilde{K}\nabla^- \tilde{w},\nabla^-(|\nabla\tilde{w}|)\rangle H'_{\epsilon}(|\nabla \tilde{w}|)\eta \tilde{\delta}dx\\
							\nonumber\int_{\Bflat^+}|\tilde{A}\nabla \tilde{w}|^{p-2}\langle \tilde{K}\nabla \tilde{w},\nabla \eta\rangle H_{\epsilon}(|\nabla \tilde{w}|) \tilde{\delta}dx\\
							\int_{\Bflat^+}|\tilde{A}\nabla \tilde{w}|^{p-2}\langle \tilde{K}\nabla \tilde{w},\nabla(|\nabla \tilde{w}|)\rangle H'_{\epsilon}(|\nabla \tilde{w}|)\eta \tilde{\delta}dx.			\end{eqnarray*}
						Taking into account that				\begin{eqnarray*}							&&\left|\int_{\Bflat^+}|\tilde{A}\nabla \tilde{w}|^{p-2}\langle \tilde{K}\nabla \tilde{w},\nabla(|\nabla \tilde{w}|)\rangle H'_{\epsilon}(|\nabla \tilde{w}|)\eta \tilde{\delta}dx\right| \ \\
							&&\leq \tilde{C} \int_{\Bflat^+}|\nabla\tilde{w}|^{p-2}|\nabla \tilde{w}| |D^2 \tilde{w}| \frac{1}{|\nabla \tilde{w}|}\chi_{(\epsilon,2\epsilon)}\eta \tilde{\delta}dx\\
							&&\displaystyle =\tilde{C}\int_{\Bflat^+}|\nabla\tilde{w}|^{p-2}|D^2 \tilde{w}|\chi_{(\epsilon,2\epsilon)}\eta \tilde{\delta}dx\stackrel{\epsilon\to 0}{\longrightarrow}0,				\end{eqnarray*}
by  \cite{CiMa} and by Lebesgue dominated Theorem.  The second integral (on $\Bflat^-$) can be treated is a similar way.  
Then, if $\epsilon \to 0$, the previous says that
						\begin{eqnarray}
							\int_{\Bflat^-\cup \Bflat^+}\tilde{g}\eta \tilde{\delta}dx=\int_{\Bflat^-}
							|\tilde{A}\nabla^- \tilde{w}|^{p-2}\langle (\tilde{K}\nabla^- \tilde{w})^-,\nabla\eta\rangle \tilde{\delta}dx\\
							\nonumber	+\int_{\Bflat^+}
							|\tilde{A}\nabla \tilde{w}|^{p-2}\langle \tilde{K}\nabla \tilde{w},\nabla \eta\rangle
							\tilde{\delta}dx,
						\end{eqnarray}
						for all $\eta\in C_c^\infty(\Bflat)$.
					\end{proof}
					
					\section{Computing the Linearized Equation}
					With a little abuse of notation,
					even if we are considering the extended functions defined on the whole $\Bflat$,
					we will omit the tilde over the functions defined in \eqref{Neqfin}. Moreover 
					by $T^\mp(z)$ we denote the traces of  the functions $z_{|_{\Bflat^\mp}}$ on the boundary $\partial S$ of the domains $\Bflat^-$ and  $\Bflat^+$ respectively.
					Finally, in order to make readable the notation, when it is possible, we redefine the constant calling them with the same name. \\
					\begin{lemma}\label{lem:comsup} Let $\Omega\subset \mathbb{R}^n$ satisfying $(H_\Omega)$,  $f\in W^{1,1}(\Omega)\cap L^q(\Omega)$ with $q>N$, $\Phi$ be the flattening operator \eqref{ft} and $\Bflat^+$ and $\Bflat^-$ as in \eqref{Bflat}.  Let $u\in C^{1,\alpha}(\bar{\Omega})$ be a weak solution of \eqref{EP}  with respect to the Dirichlet condition or the Neumann condition and $\tilde{w}$ satisfying \eqref{Neqfin}. Then for all $j=1,\dots,N$,  $\tilde{w}_j$ fulfils
\begin{eqnarray}
\nonumber &&\ib |A\nabla \tilde{w}|^{p-2}
							\langle K\nabla \tilde{w}_j,\nabla\phi\rangle \tilde{\delta}dx+\ibm |A\nabla^- \tilde{w}|^{p-2}
							\langle K\nabla^- \tilde{w}_j,\nabla^-\phi\rangle \tilde{\delta}dx\\
							\nonumber &+&\ib |A\nabla \tilde{w}|^{p-2}
							\langle K_j\nabla \tilde{w},\nabla\phi\rangle \tilde{\delta} dx +\ibm |A\nabla^- \tilde{w}|^{p-2}
							\langle K_j\nabla^- \tilde{w},\nabla^-\phi\rangle \tilde{\delta}dx\\
							\nonumber &+&
							(p-2)\ib|A\nabla \tilde{w}|^{p-4}
							\langle K\nabla \tilde{w},\nabla \tilde{w}_j\rangle
							\cdot
							\langle K\nabla \tilde{w},\nabla\phi\rangle \tilde{\delta}dx\\
							\nonumber &+&
							(p-2)\ibm|A\nabla^- \tilde{w}|^{p-4}
							\langle K\nabla^- \tilde{w},\nabla^- \tilde{w}_j\rangle
							\cdot
							\langle K\nabla^- \tilde{w},\nabla^-\phi\rangle \tilde{\delta}dx\\	
							&+&
							(p-2)\ib|A\nabla \tilde{w}|^{p-4}
							\langle A_j^T A\nabla \tilde{w},\nabla \tilde{w}\rangle
							\cdot
							\langle K\nabla \tilde{w},\nabla\phi\rangle \tilde{\delta}dx\\
							\nonumber &+&
							(p-2)\ibm|A\nabla^- \tilde{w}|^{p-4}
							\langle A_j^T A\nabla^- \tilde{w},\nabla^- \tilde{w}\rangle
							\cdot
							\langle K\nabla^- \tilde{w},\nabla^-\phi\rangle \tilde{\delta}dx\\
							\nonumber &+&	\ib |A\nabla \tilde{w}|^{p-2}
							\langle K\nabla \tilde{w},\nabla\phi\rangle \tilde{\delta}_jdx+\ibm |A\nabla^- \tilde{w}|^{p-2}
							\langle K\nabla^- \tilde{w},\nabla^-\phi\rangle \tilde{\delta}_jdx\\
							\nonumber &=&\int_{\Bflat^-\cup \Bflat^+} (g_j\tilde{\delta}+g\tilde{\delta}_j)\phi  \ dx\, \mp\delta_{jN}\int_{S}T^{\mp}(\tilde{g}\tilde{\delta}\phi) dH^{n-1},
				\end{eqnarray}
				where $\delta_{jN}$  denotes the Kronecker delta, and  $\phi\in C_c^\infty(\Bflat\setminus Z_{\tilde{w}})$.	
					\end{lemma}
\begin{remark}
We point out that Lemma \ref{lem:comsup}  actually holds for test functions $\phi\in C_c^\infty(\Bflat)$. This can be proved a posteriori via a density argument once \eqref{eq stima hessiano locale} is available. 
\end{remark}

				\begin{proof}
						Let us consider
						$\phi\in C^\infty_c(\Bflat\setminus Z_w)$. For any $j=1,...,N-1,$  using $\phi_j$ as test function
						in \eqref{Neqfin}.  Since $f$ in \eqref{EP} is in $W^{1,1}(\Omega)\cap L^q(\Omega)$ with $q>N$  and the flatting operator is $C^3$, we can integrate by parts
						obtaining that
						\begin{eqnarray}\label{linearizzato}
							\nonumber &&\ib |A\nabla \tilde{w}|^{p-2}
							\langle K\nabla \tilde{w}_j,\nabla\phi\rangle \tilde{\delta}dx+\ibm |A\nabla^- \tilde{w}|^{p-2}
							\langle K\nabla^- \tilde{w}_j,\nabla^-\phi\rangle \tilde{\delta}dx\\
							\nonumber &+&\ib |A\nabla \tilde{w}|^{p-2}
							\langle K_j\nabla \tilde{w},\nabla\phi\rangle \tilde{\delta} dx +\ibm |A\nabla^- \tilde{w}|^{p-2}
							\langle K_j\nabla^- \tilde{w},\nabla^-\phi\rangle \tilde{\delta}dx\\
							\nonumber &+&
							(p-2)\ib|A\nabla \tilde{w}|^{p-4}
							\langle K\nabla \tilde{w},\nabla \tilde{w}_j\rangle
							\cdot
							\langle K\nabla \tilde{w},\nabla\phi\rangle \tilde{\delta}dx\\
							\nonumber &+&
							(p-2)\ibm|A\nabla^- \tilde{w}|^{p-4}
							\langle K\nabla^- \tilde{w},\nabla^- \tilde{w}_j\rangle
							\cdot
							\langle K\nabla^- \tilde{w},\nabla^-\phi\rangle \tilde{\delta}dx\\
							&+&
							(p-2)\ib|A\nabla \tilde{w}|^{p-4}
							\langle A_j^T A\nabla \tilde{w},\nabla \tilde{w}\rangle
							\cdot
							\langle K\nabla \tilde{w},\nabla\phi\rangle \tilde{\delta}dx\\
							\nonumber &+&
							(p-2)\ibm|A\nabla^- \tilde{w}|^{p-4}
							\langle A_j^T A\nabla^- \tilde{w},\nabla^- \tilde{w}\rangle
							\cdot
							\langle K\nabla^- \tilde{w},\nabla^-\phi\rangle \tilde{\delta}dx\\
							\nonumber &+&	\ib |A\nabla \tilde{w}|^{p-2}
							\langle K\nabla \tilde{w},\nabla\phi\rangle \tilde{\delta}_jdx+\ibm |A\nabla^- \tilde{w}|^{p-2}
							\langle K\nabla^- \tilde{w},\nabla^-\phi\rangle \tilde{\delta}_jdx\\
							\nonumber &=&\int_{\Bflat^-\cup \Bflat^+} (g_j\tilde{\delta}+g\tilde{\delta}_j)\phi  \ dx\,.
						\end{eqnarray}
						For $j=N$ let us consider
						\begin{eqnarray}
							\int_{\Bflat^-\cup \Bflat^+}\tilde{g}\tilde{\delta}\psi_N dx=\int_{\Bflat^-}
							|\tilde{A}\nabla^- \tilde{w}|^{p-2}\langle (\tilde{K}\nabla^- \tilde{w})^-,\nabla \psi_N\rangle \tilde{\delta}dx\\
							\nonumber	+\int_{\Bflat^+}
							|\tilde{A}\nabla \tilde{w}|^{p-2}\langle \tilde{K}\nabla \tilde{w},\nabla \psi_N\rangle
							\tilde{\delta}dx\,.
						\end{eqnarray}
						Integrating by parts we notice that
						\begin{eqnarray*}
							&&\int_{\Bflat^-}
							|\tilde{A}\nabla^- \tilde{w}|^{p-2}\langle (\tilde{K}\nabla^- \tilde{w})^-,\nabla \psi_N\rangle \tilde{\delta}dx+\int_{\Bflat^+}
							|\tilde{A}\nabla \tilde{w}|^{p-2}\langle \tilde{K}\nabla \tilde{w},\nabla \psi_N\rangle \tilde{\delta}dx\\
							&=&-\int_{\Bflat^-}\langle \partial_N(\tilde{\delta}|\tilde{A}\nabla^- \tilde{w}|^{p-2}(\tilde{K}\nabla^- \tilde{w})^-),\nabla \psi\rangle dx-\int_{\Bflat^+}
							\langle\partial_N(\tilde{\delta}|\tilde{A}\nabla \tilde{w}|^{p-2} \tilde{K}\nabla \tilde{w}),\nabla \psi\rangle dx\\
							&&+\int_S|\tilde{A}\nabla^- \tilde{w}|^{p-2} \langle(\tilde{K}\nabla^- \tilde{w})^-,\nabla \psi\rangle\tilde{\delta} ds-\int_S|\tilde{A}\nabla \tilde{w}|^{p-2} \langle\tilde{K}\nabla \tilde{w},\nabla \psi\rangle\tilde{\delta} ds.
						\end{eqnarray*}
						Once more we  distinguish the case of the Dirichlet condition and the case of the Neumann condition. \\ 
						\emph{The Dirichlet condition.} By the arguments of Section \ref{sectionfermi}   and Remark \ref{remimp}, since $\nabla \tilde{w}(\bar{x},0)=(\bar{0}, \tilde{w}_N(\bar{x},0))$, $\nabla^- \tilde{w}(\bar{x},0))=(\bar{0}, -\tilde{w}_N(\bar{x},0))$ and $K|_S$ is the matrix 
						\begin{equation}
							\left(
							\begin{tabular}{llll}
								$ \ $ & $ \ $ & $ \ $ & $0$ \\
								$ \ $ & $C^{-1}(D\gamma)$ & $ \ $ & $0$ \\
								$ \ $ & $ \ $ & $ \ $ & $0$\\
								$0$ & $0$ & $\cdots \ 0$ & $1$ \\
							\end{tabular}
							\right),
						\end{equation}
						then $|\tilde{A}\nabla^- \tilde{w}|^{p-2}=|\tilde{A}\nabla \tilde{w}|^{p-2}$ outside the critical set  and the above integral on the boundary becomes
						\begin{eqnarray*}
							&& \int_S|\tilde{A}\nabla^- \tilde{w}|^{p-2} \langle(\tilde{K}\nabla^- \tilde{w})),\nabla^- \psi\rangle\tilde{\delta} ds-\int_S|\tilde{A}\nabla \tilde{w}|^{p-2} \langle\tilde{K}\nabla \tilde{w},\nabla \psi\rangle\tilde{\delta} ds\\
							&=&  \int_S|\tilde{A}\nabla \tilde{w}|^{p-2} w_{N}\psi_{N}\tilde{\delta} ds-\int_S|\tilde{A}\nabla \tilde{w}|^{p-2} w_{N}\psi_{N}\tilde{\delta} ds=0\,.
						\end{eqnarray*}\ \\
						\emph{The Neumann condition.} In this case  $\tilde{A}\nabla \tilde{w}=\tilde{A}\nabla^- \tilde{w}$; moreover by Remark \ref{remimp}, $\langle \tilde{K}\nabla^\pm \tilde{w},e_N\rangle=0$ on $S$ hence $\tilde{K}\nabla \tilde{w}=(\tilde{K}\nabla^- \tilde{w})^-$. Then 
						$$
						\int_S|\tilde{A}\nabla^- \tilde{w}|^{p-2} \langle(\tilde{K}\nabla^- \tilde{w})^-,\nabla \psi\rangle ds-\int_S|\tilde{A}\nabla \tilde{w}|^{p-2} \langle\tilde{K}\nabla \tilde{w}),\nabla \psi\rangle ds=0,
						$$
						too.\\
						To conclude let us note that we have
						\begin{eqnarray*}
							\int_{\Bflat^-\cup \Bflat^+}\tilde{g}\tilde{\delta}\psi_N dx&=&-\int_{\Bflat^-}\partial_N(\tilde{g}\tilde{\delta})\psi dx+\int_{S}T^-(\tilde{g}\tilde{\delta}\psi) ds\\
							&&-\int_{\Bflat^+}\partial_N(\tilde{g}\tilde{\delta})\psi dx-\int_{S}T^+(\tilde{g}\tilde{\delta}\psi) ds\\
							&=&-\int_{\Bflat^-\cup \Bflat^+}(g_N\tilde{\delta}+g\tilde{\delta}_N)\psi dx \pm\int_{S}T^{\mp}(\tilde{g}\tilde{\delta}\psi) dH^{n-1},
						\end{eqnarray*}
						where by $T^\mp$ we denote the traces of  the functions ${\tilde{g}\tilde{\delta}\psi}_{|_{\Bflat^\mp}}$ on the boundary $\partial S$ of the domains $\Bflat^-$ and  $\Bflat^+$ respectively.
Hence we get that			\begin{eqnarray*}
						&&	\int_{\Bflat^-\cup \Bflat^+}(g_N\tilde{\delta}+g\tilde{\delta}_N)\psi dx\mp\int_{S}T^{\mp}(\tilde{g}\tilde{\delta}\psi) dH^{n-1} 
\\&&=\int_{\Bflat^-}\langle \partial_N(|\tilde{A}\nabla^- \tilde{w}|^{p-2}(\tilde{K}\nabla^- \tilde{w})^-),\nabla^- \psi\rangle dx+\int_{\Bflat^+}
							\langle\partial_N(|\tilde{A}\nabla \tilde{w}|^{p-2} \tilde{K}\nabla \tilde{w}),\nabla \psi\rangle dx.
						\end{eqnarray*}
					\end{proof}
					
					\section{Second order estimates}
Having now at hand the linearized type equation, we are in position to deduce our second order estimates.
We have the following:
					\begin{theorem}\label{stimad2}
Let $\Omega\subset \mathbb{R}^n$ be a domain satisfying $(H_\Omega)$ and let  $f\in W^{1,1}(\Omega)\cap L^q(\Omega)$ with $q>N$.					Let ${w}\in C^{1,\alpha}(\Bflat)$ be the solution of
						\eqref{Neqfin} and let $p\in (1,\infty)$.
						For $x_0\in \Bflat$, let $r>0$ be such that $B_{2r}(x_0)\subset \Bflat$,
						there holds:
						\begin{equation}\label{eq stima hessiano locale}
							\int_{B_r(x_0)\setminus Z_w}|\nabla w|^{p-2-\beta} | D^2 w |^2dx \leq C\,,
						\end{equation}
where $C=C(x_0,r,p,N, \beta, \|w\|_{L^\infty(B_{2r})})$
 and $\beta<1$.  Consequently, the fact that the transformation is $C^2$-smooth, provides that, for some positive $\rho$:
							\begin{equation}\label{eq stima hessiano localeTTT}
							\int_{\left(B_\rho(p_0)\cap\Omega\right)\setminus Z_u}|\nabla u|^{p-2-\beta} | D^2 u |^2dx \leq C\,,
						\end{equation}
						for any $p_0\in\overline\Omega$.

					\end{theorem}
					\begin{proof} At first,  let us say that, a lot of constants are involved in our manipulations; for instance we will recall Remark \ref{rem2} and we will use the following
						\begin{eqnarray*}
&&L_\nabla:=\sup_{B_{2r}(x_0)}|\nabla w|, \qquad L_\delta:=\sup_{B_{2r}(x_0)}(|\delta|+|\nabla \delta|), \qquad l_\delta:=\inf_{B_{2r}(x_0)}|\delta|>0.
						\end{eqnarray*}
						However,  in order not to aggravate the notation,  we will opportunely redefine the constant when is needed without changing the name.
						Let $\epsilon>0$ and $G_{\epsilon}(t)$ be defined by
						$$
						G_\epsilon(t):=\left\{
						\begin{array}{ll}
							0 & \hbox{ if } |s|\leq \epsilon\\
							2t- 2\epsilon\cdot sgn(t)  & \hbox{ if }  \epsilon< |t|< 2\epsilon \\
							t & \hbox{ if } |t| \geq 2 \epsilon. 
						\end{array}
						\right.
						$$
						For any $x_0 \in \Bflat$, let $r>0$ such that $B_{2r}(x_0)\subset \Bflat$; let us fix $\varphi$ such that $\varphi=1$ su $B_{r}(x_0)$, $\displaystyle |\nabla \varphi|<\frac{2}{r}$ on $B_{2r}(x_0)\setminus B_{r}(x_0)$ and $\varphi=0$ otherwise. 
						For any  $\beta<1$ fixed
						$$
						T_{\epsilon}(t)=\frac{G_{\epsilon}(t)}{|t|^{\beta}}, \mbox{ and } \phi=T_{\epsilon}(w_j)\varphi^2.$$
						We can use $\phi$ as a test function obtaining that
						\begin{eqnarray}
							\label{i1}&\displaystyle \int_{\Bflat^\pm} |A\nabla^{\pm} w|^{p-2}|A\nabla^{\pm}w_j|^2 \ T'_{\epsilon}(w_j)\varphi^2\delta \,dx\\
							\label{i2}&\displaystyle+(p-2)\int_{\Bflat^\pm}|A\nabla^{\pm} w|^{p-4}\,  \langle  A\nabla^{\pm}w, A \nabla^{\pm} w_j \rangle^2 \ T'_{\epsilon}(w_j)\varphi^2\delta \ dx\\
							\label{i3}&=-\displaystyle\int_{\Bflat^\pm} |A\nabla^{\pm} w|^{p-2}\langle K_j  \nabla^{\pm} w,  \nabla^{\pm} T_{\epsilon}(w_j) \rangle \varphi^2\delta \  dx\ \\
							\label{i4}&	-(p-2)\displaystyle\int_{\Bflat^\pm}|A\nabla^{\pm} w|^{p-4}\,  \langle A\nabla^{\pm} w, A_j\nabla^{\pm} w \rangle\langle  K\nabla^{\pm}w,  \nabla^{\pm}T_{\epsilon}(w_j) \rangle \varphi^2\delta \ dx\\
							\label{i5}&\displaystyle -\int_{\Bflat^\pm}  |A\nabla^{\pm} w|^{p-2}
							\langle K\nabla^{\pm} w,\nabla^{\pm} T_{\epsilon}(w_j) \rangle  \varphi^2 \delta_jdx\\
							\label{i6}&-2\displaystyle\int_{\Bflat^\pm} |A\nabla^{\pm}  w|^{p-2}\langle K\nabla^{\pm}  w_j, \nabla^{\pm}\varphi\rangle T_{\epsilon}(w_j)\varphi \,\delta dx\\
							\label{i7}&-2\displaystyle\int_{\Bflat^\pm} |A\nabla^{\pm} w|^{p-2}\langle K_j  \nabla^{\pm} w,  \nabla^{\pm} \varphi\rangle T_{\epsilon}(w_j)\varphi\delta \ dx\\
							\label{i8}&-(2p-4)\displaystyle\int_{\Bflat^\pm}|A\nabla^{\pm}  w|^{p-4} \langle K  \nabla^{\pm} w,  \nabla^{\pm} w_j \rangle \langle K\nabla^{\pm}  w,  \nabla^{\pm} \varphi \rangle T_{\epsilon}(w_j)\varphi\delta \ dx\\
							\label{i9}&-(2p-4)\displaystyle\int_{\Bflat^\pm}|A\nabla^{\pm} w|^{p-4} \langle A\nabla^{\pm}  w, A_j\nabla^{\pm}  w \rangle \langle K\nabla^{\pm}  w,  \nabla^{\pm} \varphi \rangle  T_{\epsilon}(w_j)\varphi\delta \ dx\\
							\label{i10}&\displaystyle -2\int_{\Bflat^\pm}  |A\nabla^{\pm} w|^{p-2}
							\langle K\nabla^{\pm} w,\nabla^{\pm} \varphi\rangle  T_{\epsilon}(w_j)\varphi\delta_jdx\\
							\label{i11}&\displaystyle+\int_{\Bflat^-\cup \Bflat^+}[g_j\delta+g\delta_j]T_{\epsilon}(w_j)\varphi^2\,dx \mp\delta_{jN}\int_{S}T^{\mp}(\tilde{g}\tilde{\delta}T_{\epsilon}(w_j)\varphi^2) dH^{n-1} .
						\end{eqnarray}
						where each integral on $\Bflat^+$ involves the classical $\nabla$ whereas each integral on $\Bflat^-$ involves $\nabla^-$.  Note that to deduce the previous equation we need to argue by density to plug in the test functions. This is standard but we remark that for the two terms in \eqref{i11} it is required to exploit the dominate convergence Theorem and the fact that $T_\varepsilon (w_j)$ is  bounded. Let us now focus on \eqref{i1} and \eqref{i2}; note that for $p>2$, \eqref{i2}$\geq0$ and $\eqref{i1}+\eqref{i2}\geq \eqref{i2}$. Similarly, when $p<2$
						\begin{eqnarray*}
							&&(p-2)\int_{\Bflat^\pm}|A\nabla^{\pm} w|^{p-4}\,  \langle  A\nabla^{\pm}w, A \nabla^{\pm} w_j \rangle^2  T'_{\epsilon}(w_j)\varphi^2\delta \ dx\\
							&\geq& (p-2)\int_{\Bflat^\pm}|A\nabla^{\pm} w|^{p-2}\,  |A \nabla^{\pm} w_j|^2 T'_{\epsilon}(w_j)\varphi^2\delta \  dx
						\end{eqnarray*}
						and then
						\begin{eqnarray*}
							\eqref{i1}+\eqref{i2}&\geq& \min\{1,(p-1)\}\int_{\Bflat^\pm} |A\nabla^{\pm} w|^{p-2}|A\nabla^{\pm}w_j|^2 T'_{\epsilon}(w_j)\varphi^2 \delta \ dx\\
							&\geq& c_A^p\min\{1,(p-1)\}\int_{\Bflat^\pm} |\nabla w|^{p-2}|\nabla w_j|^2 T'_{\epsilon}(w_j)\varphi^2 \delta \ dx\\
							&\geq& c_A^pl_\delta\min\{1,(p-1)\}\int_{\Bflat^\pm} \frac{|\nabla w|^{p-2}|\nabla w_j|^2\varphi^2} {|w_j|^\beta}\left(G'_\epsilon(w_j)-\beta\frac{G_\epsilon(w_j)}{|w_j|}\right)\ dx\,.
						\end{eqnarray*}
Note now that the estimates from above in  Remark \ref{rem2}  can be deduced also for  $A_j$ and $K_j$. Therefore there exists $\tilde{C_\gamma}>0$ such that						\begin{eqnarray*}
							|\eqref{i6}|+\ldots+|\eqref{i10}|&\leq& \tilde{C_\gamma}\left( \displaystyle\int_{\Bflat^\pm} \varphi|\nabla w|^{p-2}|\nabla w_j|  |\nabla \varphi| |T_{\epsilon}(w_j)|   \delta\ dx \right. \ \\
							&&+\left. \displaystyle\int_{\Bflat^\pm} \varphi|\nabla w|^{p-1} |\nabla \varphi| |T_{\epsilon}(w_j)| (\delta+|\delta_j|) \ dx \right)\,.
						\end{eqnarray*}
Since $$\displaystyle |\nabla \varphi|\leq \frac{2}{r} \quad\text{and}\quad |T_{\epsilon}(w_j)| \leq |w_j|^{1-\beta}\leq |\nabla w |^{1-\beta},$$  by weighted Young inequality, we get		\begin{eqnarray*}
							&& \displaystyle  \int_{\Bflat^\pm} \varphi|\nabla w|^{p-2}|\nabla w_j|  |\nabla \varphi| |T_{\epsilon}(w_j)|  \delta\ dx+ \displaystyle \int_{\Bflat^\pm} \varphi|\nabla w|^{p-1} |\nabla \varphi| |T_{\epsilon}(w_j)|(\delta+|\delta_j|)  \ dx \\
							&\leq&\frac{2L_\delta}{r}\left(\displaystyle  \int_{\Bflat^\pm\setminus Z_w} \varphi|\nabla w|^{p-2}|\nabla w_j|   |w_j|^{1-\beta} \,dx+ \displaystyle \int_{\Bflat^\pm\setminus Z_w} \varphi|\nabla w|^{p-\beta}\,dx\right) \\
							&\leq& \frac{2L_\delta\theta}{r}\displaystyle  \int_{\Bflat^\pm\setminus Z_w} \frac{\varphi^2|\nabla w|^{p-2}|\nabla w_j| ^2}{ |w_j|^{\beta}}\chi_{\{|w_j|>\epsilon\}} \,dx+\frac{(1+4C_\varphi\theta)L_\delta}{2r\theta}\displaystyle \int_{\Bflat^\pm\setminus Z_w} |\nabla w|^{p-\beta}\,dx,
						\end{eqnarray*}
						where $\displaystyle C_\varphi:=\max_{supp \ \varphi}(\varphi)$. For a fixed constant  $D_1:=D_1(r, \theta, L_\nabla, L_\delta, C_\varphi)$,  since $p>1>\beta$, we conclude that
						\begin{eqnarray*}
							\displaystyle  &&\int_{\Bflat^\pm} \varphi|\nabla w|^{p-2}|\nabla w_j|  |\nabla \varphi| |T_{\epsilon}(w_j)|\delta \ dx+ \displaystyle \int_{\Bflat^\pm} \varphi|\nabla w|^{p-1} |\nabla \varphi| |T_{\epsilon}(w_j)| (\delta+|\delta_j|)\ dx  \\
							&\leq& \frac{2L_\delta\theta}{r}\displaystyle  \int_{\Bflat^\pm\setminus Z_w} \frac{\varphi^2|\nabla w|^{p-2}|\nabla w_j| ^2}{ |w_j|^{\beta}}\chi_{\{|w_j|>\epsilon\}} \,dx+D_1\,.
						\end{eqnarray*}
						By the regularity properties of $g$,  since $\beta<1$, there exists $c_g=c_g(L_\delta, L_\nabla, C_\varphi, g)$ positive,  such that
						\begin{eqnarray*}
							|\eqref{i11}|\leq \int_{\Bflat^-\cup \Bflat^+}(|g_j|\delta+|g\delta_j|)|w_j|^{1-\beta}\varphi^2\,dx +\int_{S}|T^{\mp}(\tilde{g}\tilde{\delta}T_{\epsilon}(w_j)\varphi^2) |dH^{n-1}<c_g\,.
						\end{eqnarray*}
						Taking into account this last, redefining the suitable constants, we get the following:
						\begin{eqnarray}\label{lunga}
							\nonumber &&c_A^pl_\delta\min\{1,(p-1)\}\int_{\Bflat^\pm} \frac{|\nabla w|^{p-2}|\nabla w_j|^2\varphi^2} {|w_j|^\beta}\left(G'_\epsilon(w_j)-\beta\frac{G_\epsilon(w_j)}{|w_j|}\right)\ dx\\ &\leq&  \eqref{i3}+\eqref{i4}+\eqref{i5}\\					\nonumber&&+\frac{2L_\delta\tilde{C}_\gamma\theta}{r}\displaystyle  \int_{\Bflat^\pm\setminus Z_w} \frac{\varphi^2|\nabla w|^{p-2}|\nabla w_j| ^2 }{ |w_j|^{\beta}}\chi_{\{|w_j|>\epsilon\}} \,dx+D_{1,g}\,.
						\end{eqnarray}
						Regarding the estimate of $\eqref{i3}$-$\eqref{i5}$,  we can use a similar approach; therefore let us write the general form
						$$
						\int_{\Bflat^\pm}h(x)\langle M  \nabla^{\pm} w,  \nabla^{\pm} T_{\epsilon}(w_j) \rangle\varphi^2\,dx,
						$$
						where $M:=M(x)$ is a given $n$-matrix ($K$ or $K_j$ respectively) and $h(x)$ is
						\begin{eqnarray*}
						 &&h(x)=|A\nabla^{\pm} w|^{p-2}\delta \qquad \mbox{ in \eqref{i3}}\\
						 &&h(x)=|A\nabla^{\pm} w|^{p-4}\,  \langle A\nabla^{\pm} w, A_j\nabla^{\pm} w \rangle \ \delta\qquad \mbox{ in \eqref{i4}}\\
						 &&h(x)=|A\nabla^{\pm} w|^{p-2}\delta_j\qquad \mbox{ in \eqref{i5}}\,.
						\end{eqnarray*}
						Since $\varphi=0$ on $^cB_{2r}(x_0)$ (see also the argument used in  \eqref{bordodopo} and \eqref{bordodopo2}), integrating by part one can see that
						\begin{eqnarray*}
							\int_{\Bflat^\pm} h\varphi^2\langle K\nabla^{\pm} w,  \nabla^{\pm} T_{\epsilon}(w_j) \rangle  \,dx
							&=&-\int_{S} h\varphi^2 T_{\epsilon}(w_j)\langle K\nabla w,  e_N \rangle  \,ds\\
							&&+\int_{S} h\varphi^2 T_{\epsilon}(w_j)\langle (K\nabla^{-}w)^-,  e_N \rangle  \,ds\\
							&&-\int_{\Bflat^\pm} div(h\varphi^2 (K\nabla^{\pm} w)^\pm) T_{\epsilon}(w_j) \ dx\\
							&=&-\int_{\Bflat^\pm} div(h\varphi^2 (K\nabla^{\pm} w)^\pm) T_{\epsilon}(w_j) \ dx.				\end{eqnarray*}
					More details are needed to see that, for $j=1,\dots,N$
						\begin{equation*}\label{soomabordi}
							\int_{\Bflat^\pm} h\varphi^2\langle K_j\nabla^{\pm} w,  \nabla^{\pm} T_{\epsilon}(w_j) \rangle  \,dx
							=-\int_{\Bflat^\pm} div(h\varphi^2 (K_j\nabla^{\pm} w)^\pm) T_{\epsilon}(w_j) \ dx,
						\end{equation*}
						i.e. to see that
						\begin{equation}\label{soomabordi}
							-\int_{S} h\varphi^2 T_{\epsilon}(w_j)\langle K_j\nabla w,  e_N \rangle  \,ds
							+\int_{S} h\varphi^2T_{\epsilon}(w_j)\langle (K_j\nabla^{-}w)^-,  e_N \rangle  \,ds=0\,.
						\end{equation}
As above we need to distinguish if the Dirichlet or the Neumann condition hold for \eqref{EP}. Under the Neumann condition, for $j=N$ it is enough to note that on $S$, $T_\epsilon(w_N)=0$. For $j=1,\dots, N-1$, one can see by definition of $K$ that
						$$
						\partial_j (K)|_S=-K|_S(\partial_jK^{-1})|_SK|_S,
						$$
						where $K|_S$ can be  deduced by \eqref{kappaBordo} and $\partial_jK^{-1}|_S$ can be obtained by \eqref{kappagen}.
						Therefore is only a technical argument to see that on $S$
						$$\langle K_j\nabla^\pm w,e_N\rangle=\langle (\partial_j K_{(N)})\nabla^\pm w,e_N\rangle=0,$$ 
						(where $K_{(N)}$ represents the last row of the matrix) and \eqref{soomabordi} holds.
						Taking into account Dirichlet conditions, let us immediately notes that for $j=1,\ldots,N-1$, $w_j|_S=0$ therefore $T_\epsilon(w_j)=0$. For $j=N$, 
						$\langle K_N\nabla w,e_N\rangle=\langle \partial_NK_{(N)},\nabla w\rangle$, where $K_{(N)}$ represents the $N$-row of $K$. Then by Remark \ref{remimp}, $\nabla^\pm w|_S=(\bar{0},\pm w_N)$ and we obtain that
						\begin{eqnarray*}
							&&-\int_{S} h\varphi^2 T_{\epsilon}(w_N)\langle K_N\nabla w,  e_N \rangle ds
							+\int_{S} h\varphi^2T_{\epsilon}(w_N)\langle (K_N\nabla^{-}w)^-,  e_N \rangle ds\\
							&=&-\int_{S} h\varphi^2T_{\epsilon}(w_N)\langle \partial_NK_{(N)},\nabla w\rangle ds
							-\int_{S} h\varphi^2T_{\epsilon}(w_N)\langle \partial_NK_{(N)},\nabla^- w\rangle \,ds\\
							&=&-\int_{S} h\varphi^2T_{\epsilon}(w_N)w_N (\partial_N K_{(N,N)}) ds
							+\int_{S} h\varphi^2T_{\epsilon}(w_N)w_N (\partial_N K_{(N,N)}) \,ds,
						\end{eqnarray*}
						(where $K_{(N,N)}$ is the element of the matrix in $N$-row and $N$-column).
\begin{remark}	There is here a little abuse of notations in the case 	
\[h(x)=|A\nabla^{\pm} w|^{p-2}\delta_j\qquad \mbox{ in \eqref{i5}}\,.
\]	
In fact, when integrating by parts the derivatives of $\delta$ on the boundary may have different values depending on which integration domain $\Bflat^+$ or $\Bflat^-$ we are considering. This will make no differences in the computations here below since we actually prove that all these terms are zero.			
\end{remark}

						 Let us prove that $\partial_N K_{(N,N)}|_S=0$.
						In a similar way to what we saw for the Neumann condition, notice that
						$$
						 \partial_N (K^{-1})|_S=-K^{-1}|_S(\partial_NK|_S)K^{-1}|_S\,.
						$$
						Now, $K^{-1}|_S$ can be deduced by \eqref{kappaBordo} and direct computation gives us that $\partial_N (K_{(N,N)}^{-1})=0$. Then
						$$
						 0=\partial_N (K_{(N,N)}^{-1})|_S=\langle e_N(\partial_NK|_S),e_N\rangle=\partial_N K_{(N,N)}|_S
						$$
						hence the claim.
						Therefore  we can compute that,
						\begin{eqnarray*}
							&&\int_{\Bflat^\pm} h\varphi^2\langle M\nabla^{\pm} w,  \nabla^{\pm} T_{\epsilon}(w_j) \rangle  \,dx
							=-\int_{\Bflat^\pm} div( h\varphi^2 (M\nabla^{\pm} w)^\pm) T_{\epsilon}(w_j) \ dx\\
							&=&-\int_{\Bflat^\pm} h\varphi^2div((M\nabla^{\pm} w)^\pm) T_{\epsilon}(w_j) \ dx-\int_{\Bflat^\pm} \langle M\nabla^{\pm} w,  \nabla^{\pm}  (h\varphi^2) \rangle T_{\epsilon}(w_j) \,dx\\
							&=&-\int_{\Bflat^\pm} h\varphi^2div((M\nabla^{\pm} w)^\pm) T_{\epsilon}(w_j) \ dx-\int_{\Bflat^\pm}2\varphi  h\langle M\nabla^{\pm} w,  \nabla^{\pm}  \varphi \rangle T_{\epsilon}(w_j) \,dx\\
							&&-\int_{\Bflat^\pm}\varphi^2 \langle M\nabla^{\pm} w,  \nabla^{\pm} h\rangle T_{\epsilon}(w_j) \,dx\\
							&\leq&\int_{\Bflat^\pm} \varphi^2|h| |div((M\nabla^{\pm} w)^\pm)| |w_j|^{1-\beta} \chi_{\{|w_j|>\epsilon\}}\, dx\\
							&&+\frac{2}{r}\int_{\Bflat^\pm}\varphi |h||M\nabla^{\pm} w| |w_j|^{1-\beta} \chi_{\{|w_j|>\epsilon\}}\,dx\\
							&&+\int_{\Bflat^\pm}\varphi^2| M\nabla^{\pm} w| | \nabla  h| |w_j|^{1-\beta} \chi_{\{|w_j|>\epsilon\}}\,dx\,.
						\end{eqnarray*}
						In our case, $h= |A\nabla^{\pm} w|^{p-2}\delta$, or $h= |A\nabla^{\pm} w|^{p-2}\delta_j$ or $h=\delta |A\nabla^{\pm} w|^{p-4}\,  \langle A\nabla^{\pm} w, A_j\nabla^{\pm} w \rangle$. Therefore there exist a suitable $c_{\gamma,\delta}>0$ such that $|h|\leq c_{\gamma,\delta}|\nabla w|^{p-2}$. In this way, 
						$$
						\frac{2}{r}\int_{\Bflat^\pm}\varphi |h||M\nabla^{\pm} w| |w_j|^{1-\beta}  \chi_{\{|w_j|>\epsilon\}} \,dx\leq \frac{2c_{\gamma,\delta}}{r}\int_{\Bflat^\pm\setminus Z_w}\varphi |\nabla w|^{p-1}|w_j|^{1-\beta} \,dx:=d_1(>0).
						$$
						Next we estimate $|\nabla h|$. For $h= |A\nabla^{\pm} w|^{p-2}\delta$ or $h= |A\nabla^{\pm} w|^{p-2}\delta_j$ then
						$$
						\nabla^{\pm} h=\delta (p-2)| |A\nabla^{\pm} w|^{p-4} [(\langle A\nabla^{\pm} w, A_k\nabla^{\pm} w\rangle)_{k}+(\langle K\nabla^{\pm} w,\nabla^{\pm} w_k\rangle)_{k}]+|A\nabla^{\pm} w|^{p-2}\nabla \delta
						$$
						or
						$$
						\nabla^{\pm} h=\delta_j(p-2)| |A\nabla^{\pm} w|^{p-4} [(\langle A\nabla^{\pm} w, A_k\nabla^{\pm} w\rangle)_{k}+(\langle K\nabla^{\pm} w,\nabla^{\pm} w_k\rangle)_{k}]+|A\nabla^{\pm} w|^{p-2}\nabla \delta_j
						$$
						hence,  passing to the norm, taking into account Remark \ref{rem2}, for both cases we can define $\tilde{c_1}:=\tilde{c_1}(p,L_\delta,C_A,C_K)$ and $\bar{c_2}:=\bar{c_2}(p,L_\delta, C_A,C_K)$ such that
						$$
						|\nabla h|\leq \tilde{c_1} |\nabla w|^{p-2}+\bar{c_2}|\nabla w|^{p-3}|D^2w|
						$$
						The same inequality can be obtained considering $h=\delta|A\nabla^{\pm} w|^{p-4}\,  \langle A\nabla^{\pm} w, A_j\nabla^{\pm} w \rangle$.
						Therefore in both cases, since $\beta<1$ and by using the $W^{1,2}(\Omega)$-regularity (see Remark~\ref{remregcianchi}) the exists a positive $C_w$ constant such that
						\begin{eqnarray*}
							\int_{\Bflat^\pm}\varphi^2 | M\nabla^{\pm} w| | \nabla  h| |w_j|^{1-\beta}  \chi_{\{|w_j|>\epsilon\}}\,dx&\leq&  \tilde{c_1}\int_{\Bflat^\pm\setminus Z_w}\varphi^2| \nabla w|^{p-1}|w_j|^{1-
								\beta} \,dx\\
							&&+\bar{c_2}\int_{\Bflat^\pm\setminus Z_w}\varphi^2|\nabla w|^{p-2} |D^2w| |w_j|^{1-\beta} \,dx\\
							&\leq& C_w.
						\end{eqnarray*}
						In the end, denoting  by $M_{(i)}$ the $i$-row in the matrix $M$ we have
						\begin{eqnarray*}
							|div((M\nabla^{\pm} w)^\pm)|&=&\left|\sum_{i=1}^{N-1}\partial_{x_i}\langle M_{(i)},\nabla^\pm w\rangle\pm\partial_{x_N}\langle M_{(N)},\nabla^\pm w\rangle\right|\\
							&\leq&\sum_{i=1}^{N-1}|\langle M_{(i),i},\nabla w^\pm\rangle|+\sum_{i=1}^{N-1}|\langle M_{(i)},\nabla^{\pm} w_i\rangle|\\
							&&+|\langle M_{(N),N},\nabla w^\pm\rangle|+|\langle M_{(N)},\nabla^{\pm} w_N\rangle|.
						\end{eqnarray*}
Therefore, since $M$ represents $K$ or $K_j$ respectively, there exists a constant $d_K>0$ such that
						\begin{eqnarray*}
							&&\int_{\Bflat^\pm} \varphi^2|h||div((M\nabla^{\pm} w)^\pm)| |w_j|^{1-\beta} \chi_{\{|w_j|>\epsilon\}}\, dx\\&&\leq c_\gamma \int_{\Bflat^\pm} \varphi^2|\nabla^{\pm} w|^{p-2} |div((M\nabla^{\pm} w)^\pm)| |w_j|^{1-\beta} \chi_{\{|w_j|>\epsilon\}}\, \ dx\\
							&&\leq d_K\ \int_{\Bflat^\pm} \varphi^2|\nabla w|^{p-\beta} \ dx+d_K\int_{\Bflat^\pm} \varphi^2|\nabla w|^{p-2} |D^2 w| |w_j|^{1-\beta} \chi_{\{|w_j|>\epsilon\}}\, \ dx\,.
						\end{eqnarray*}
						As above, by the $W^{1,2}(\Omega)$-regularity  we have						$$
						 \int_{\Bflat^\pm} \varphi^2|h||div(M\nabla^{\pm} w)| |w_j|^{1-\beta} \ dx\leq C.
						$$
						Summarizing up for \eqref{i3}-\eqref{i5} we obtain that there exists $\tilde{d}>0$ such that
						$$ 
						|\eqref{i3}|+|\eqref{i4}|+|\eqref{i4}|\leq \tilde{d}.
						$$
						Replacing in $\eqref{lunga}$ we obtain
						
						\begin{eqnarray*}
							&&c_A^pl_\delta\min\{1,(p-1)\}\int_{\Bflat^\pm\setminus Z_w} \frac{|\nabla w|^{p-2}|\nabla w_j|^2\varphi^2} {|w_j|^\beta}\left(G'_\epsilon(w_j)-\beta\frac{G_\epsilon(w_j)}{|w_j|}\right) \ dx\\
							&\leq&\tilde{d}+\frac{2L_\delta\tilde{C}_\gamma\theta}{r}\displaystyle  \int_{\Bflat^\pm\setminus Z_w} \frac{|\nabla w|^{p-2}|\nabla w_j| ^2 \varphi^2}{ |w_j|^{\beta}}\chi_{\{|w_j|>\epsilon\}} \,dx+D_{1,g},
						\end{eqnarray*}
					i.e.
						\begin{eqnarray*}
							\int_{\Bflat^\pm\setminus Z_w} \frac{|\nabla w|^{p-2}|\nabla w_j|^2\varphi^2} {|w_j|^\beta}\Bigg[c_A^pl_\delta\min\{1,(p-1)\}\left (G'_\epsilon(w_j)-\beta\frac{G_\epsilon(w_j)}{|w_j|}\right)\ldots\\\ldots-\frac{2L_\delta\tilde{C}_\gamma\theta}{r}\chi_{\{|w_j|>\epsilon\}}\Bigg] \ dx
						\leq D\,.
					\end{eqnarray*}
						Let now $\theta$ be such that 
$$						\left[c_A^pl_\delta\min\{1,(p-1)\}\left (G'_\epsilon(w_j)-\beta\frac{G_\epsilon(w_j)}{|w_j|}\right)-\frac{2L_\delta\tilde{C}_\gamma\theta}{r}\chi_{\{|w_j|>\epsilon\}}\right]>0.
						$$
						Moreover 
					\begin{eqnarray*}
					\left[c_A^pl_\delta\min\{1,(p-1)\}\left (G'_\epsilon(w_j)-\beta\frac{G_\epsilon(w_j)}{|w_j|}\right)-\frac{2L_\delta\tilde{C}_\gamma\theta}{r}\chi_{\{|w_j|>\epsilon\}}\right]\\
					\to c_A^pl_\delta\min\{1,(p-1)\}\left (1-\beta\right)-\frac{2L_\delta\tilde{C}_\gamma\theta}{r}>0,				\end{eqnarray*}
						as $\epsilon \to 0$ and $\theta$ small. By Fatou Lemma
						$$
						\int_{\Bflat^\pm\setminus Z_w} \frac{\varphi^2|\nabla w|^{p-2} |\nabla w_j|^2}{|w_i|^\beta}dx<+\infty
						$$
						and therefore
						$$
						\int_{B_r(x_0)\setminus Z_w} |\nabla w|^{p-2-\beta} |\nabla w_j|^2<+\infty\,.
						$$
						The proof of \eqref{eq stima hessiano localeTTT} follows now by direct, quite laborious,  computation. We avoid it and we only point out the  fact that the map arising from the Fermi coordinates is of class $C^2$ since the domain is of class $C^3$. 
					\end{proof}

					\section{Proof of the main results}
					We are now ready to prove our main results.
					\begin{proof}[Proof of Theorem \ref{Reg}]
						Let $\varepsilon >0$ be  fixed and let  $u$   be a solution to our problems.  Set 
						$$
						V_{\varepsilon,i}=V_{\varepsilon,i}(u):=|\nabla u|^{(\alpha-1)} G_{\varepsilon}(u_i).
						$$
						Since
						\begin{equation*}\nabla V_{\varepsilon,i}=
							\begin{cases} \displaystyle
								(\alpha-1)|\nabla u|^{(\alpha-2)}\nabla(|\nabla u|) G_{\varepsilon}(u_i)+|\nabla u|^{(\alpha-1)} G'_{\varepsilon}(u_i)\nabla u_i& |u_i|\geq \varepsilon\\
								0 & \text{otherwise},
							\end{cases}
						\end{equation*}
						for a suitable $C=C(\alpha)$, by using Theorem \ref{stimad2} with $\beta=p-2\alpha$, we have that 
						$$
						\int_{\Omega} |\nabla V_{\varepsilon,i}|^2\leq C\int_{\Omega\setminus Z_u}|\nabla u|^{2(\alpha-1)}|D^2u|^2<\infty,
						$$
						hence $V_{\varepsilon,i}\in W^{1,2}(\Omega)$ and it is uniformly bounded in the space.
There exists $\tilde{V}\in W^{1,2}(\Omega)$ such that
$$
 V_{\varepsilon,i}\rightharpoonup \tilde{V}\quad \text{as}\quad\varepsilon \rightarrow 0.
$$
By  the compact embedding $V_{\varepsilon,i}\to \tilde{V}$ in $L^q$ with $q<2^*$ and up to subsequence  $V_{\varepsilon,i}\to \tilde{V}$ a.e..  $V_{\varepsilon,i}\to |\nabla u|^{(\alpha-1)}u_i$ a.e.  therefore
$$
 \tilde{V}=|\nabla u|^{(\alpha-1)}u_i\in W^{1,2}(\Omega),
$$
for  $i=1,\ldots, N$. Then $V_\alpha=|\nabla u|^{(\alpha-1)}\nabla u\in W^{1,2}(\Omega)$.

To conclude the proof, note that for $p<3$ we can choose $\alpha=1$ getting that $u\in W^{2,2}(\Omega)$.
\end{proof}
Let us now prove an important result regarding the summability of the weight. 

\begin{proposition}\label{Reg2}
	Let $u\in C^{1,\alpha}(\bar{\Omega})$ be a weak solution of \eqref{EP}  under  \textit{Dirichlet condition} or \textit{Neumann condition} with  $\Omega\subset \mathbb{R}^n$  satisfying $(H_\Omega)$
	and
	\[
	f\in W^{1,1}(\Omega)\cap L^q(\Omega), \,\, q>N\quad with \quad  f\leq -\tau <0\,\, \text{or}\,\, f\geq \tau >0 \quad a.e. \,\,in \quad  \Omega\,.
	\]
	Then 
	\[\int_{\Omega}\frac{1}{|\nabla u|^{(p-1)\sigma}}\leq C,\] where $\sigma <1$.
\end{proposition}

\begin{proof}
	Without loss of generality, we may, and do, reduce to work in the transformed problem. 
Let therefore $\varphi$ be defined as in Theorem \ref{stimad2}  and let $\varepsilon>0$ and $\alpha<1$. 
Let us consider$$
\psi:=\frac{\varphi^2}{(\varepsilon+|\nabla \tilde w|^{p-1})^\sigma}.
$$   
\begin{remark}
We actually exploit the fact that $|\nabla \tilde w|^{p-1}\in W^{1,2}(\Bflat^+)$ and, by the fact that it is even with respect to the $y_N$-direction, actually $$|\nabla \tilde w|^{p-1}\in W^{1,2}(\Bflat).$$ Consequently by a standard density argument $\psi$ can be used as test function in~\eqref{Neqfin}. \end{remark}
Using equation \eqref{Neqfin} we obtain 
 \begin{eqnarray*}
 &&\int_{\Bflat^\pm}\frac{\tilde g\varphi^2\tilde{\delta}}{(\varepsilon+|\nabla \tilde w|^{p-1})^\sigma}=\int_{\Bflat^\pm}
  |\tilde{A}\nabla^\pm \tilde{w}|^{p-2}\langle (\tilde{K}\nabla^\pm \tilde{w})^{\pm},\nabla \psi\rangle \tilde{\delta}dx\\
  \nonumber	&=&2\int_{\Bflat^\pm}
 \frac{ |\tilde{A}\nabla^\pm \tilde{w}|^{p-2}\varphi}{(\varepsilon+|\nabla \tilde w|^{p-1})^\sigma}\langle (\tilde{K}\nabla^\pm \tilde{w})^{\pm},\nabla \varphi\rangle \tilde{\delta}dx\\
 \nonumber	&&-\sigma(p-1)\int_{\Bflat^\pm}
  \frac{\varphi^2|\tilde{A}\nabla^\pm \tilde{w}|^{p-2}|\nabla\tilde{w}|^{p-2}}{(\varepsilon+|\nabla \tilde w|^{p-1})^{\sigma+1}}\langle (\tilde{K}\nabla^\pm \tilde{w})^{\pm},\nabla |\nabla\tilde{w}|\rangle \tilde{\delta}dx.
\end{eqnarray*}
Now if $\tilde g$ is positive we set 
$\displaystyle \inf_{B_{2\rho(x_0)}}\tilde g:=m_{\tilde g}>0$, if else $ \tilde g$ is negative we set $\displaystyle \inf_{B_{2\rho(x_0)}}-\tilde g:=m_{\tilde g}>0.$
 Denoting  by $\displaystyle l_{\tilde\delta}:=\inf_{B_{2r}(x_0)}|\tilde\delta|>0$ then
\begin{eqnarray}\label{invpe}
	&&m_{\tilde g}l_{\tilde\delta}\int_{\Bflat^\pm}\frac{\varphi^2}{(\varepsilon+|\nabla \tilde w|^{p-1})^\sigma}\leq2\int_{\Bflat^\pm}
	\frac{ |\tilde{A}\nabla^\pm \tilde{w}|^{p-2}\varphi}{(\varepsilon+|\nabla \tilde w|^{p-1})^\sigma}\langle (\tilde{K}\nabla^\pm \tilde{w})^{\pm},\nabla \varphi\rangle \tilde{\delta}dx\\
	\nonumber	&&-\sigma(p-1)\int_{\Bflat^\pm}
	\frac{\varphi^2|\tilde{A}\nabla^\pm \tilde{w}|^{p-2}|\nabla\tilde{w}|^{p-2}}{(\varepsilon+|\nabla  \tilde w|^{p-1})^{\sigma+1}}\langle (\tilde{K}\nabla^\pm \tilde{w})^{\pm},\nabla |\nabla\tilde{w}|\rangle \tilde{\delta}dx.
\end{eqnarray}
Note that, the case $\tilde g$ negative, reduces to the positive case just changing sign in the equation.

It is now  an easy computation to see that the first integral on the right is finite (uniformly in $\varepsilon$), namely
\begin{eqnarray*}
 \int_{\Bflat^\pm} \frac{ |\tilde{A}\nabla^\pm \tilde{w}|^{p-2}\varphi}{(\varepsilon+|\nabla \tilde w|^{p-1})^\sigma}\langle (\tilde{K}\nabla^\pm \tilde{w})^{\pm},\nabla \varphi\rangle \tilde{\delta}dx&\leq& \tilde{C}\int_{\Bflat^\pm}\frac{|\nabla\tilde{w}|^{p-1}\varphi}{(\varepsilon+|\nabla \tilde w|^{p-1})^\sigma}|\nabla \varphi|dx\\&\leq& \frac{2\tilde{C}}{\rho}\int_{\Bflat^\pm}|\nabla\tilde{w}|^{(p-1)(1-\sigma)}\varphi dx\leq C_\rho.
\end{eqnarray*}
In order to give an estimate of the second integral, we will use the Young inequality. Then 
\begin{eqnarray*}&&
\int_{\Bflat^\pm}
\frac{\varphi^2|\tilde{A}\nabla^\pm \tilde{w}|^{p-2}|\nabla\tilde{w}^{p-2}|}{(\varepsilon+|\nabla \tilde w|^{p-1})^{\sigma+1}}\langle (\tilde{K}\nabla^\pm \tilde{w})^{\pm},\nabla |\nabla\tilde{w}|\rangle \tilde{\delta}dx\\&&\leq \bar{C} \int_{\Bflat^\pm}\frac{\varphi^2|\nabla\tilde{w}|^{p-2}|\nabla\tilde{w}|^{p-1}|D^2\tilde{w}|}{(\varepsilon+|\nabla \tilde w|^{p-1})^{\sigma+1}}dx\leq \bar{C} \int_{\Bflat^\pm}\frac{\varphi^2|\nabla\tilde{w}|^{p-2}|D^2\tilde{w}|}{(\varepsilon+|\nabla \tilde w|^{p-1})^{\sigma}}dx
\\
&&\leq \bar{C}\theta \int_{\Bflat^\pm}\frac{\varphi^2}{(\varepsilon+|\nabla \tilde w|^{p-1})^{\sigma}}dx+\frac{\bar{C}}{4\theta}\int_{\Bflat^\pm} \varphi^2|\nabla\tilde{w}|^{p-2}|D^2\tilde{w}|^2\frac{|\nabla\tilde{w}|^{p-2}}{(\varepsilon+|\nabla \tilde w|^{p-1})^{\sigma}}dx\\
&&\leq \bar{C}\theta \int_{\Bflat^\pm}\frac{\varphi^2}{(\varepsilon+|\nabla \tilde w|^{p-1})^{\sigma}}dx+\frac{\bar{C}}{4\theta}\int_{\Bflat^\pm}\varphi^2|\nabla\tilde{w}|^{p-2-\beta}|D^2\tilde{w}|^2dx
\end{eqnarray*}
where $\beta=(p-1)\sigma-(p-2)$, from which one deduces $\beta<1$. Hence we can apply Theorem \eqref{stimad2}.
Back to \eqref{invpe} we get
\begin{eqnarray}
	m_gl_\delta\int_{\Bflat^\pm}\frac{\varphi^2dx}{(\varepsilon+|\nabla \tilde w|^{p-1})^\sigma}&\leq&\bar{C}\theta \int_{\Bflat^\pm}\frac{\varphi^2dx}{(\varepsilon+|\nabla \tilde w|^{p-1})^{\sigma}}dx+\bar{D}.
\end{eqnarray}
Choosing $\theta$ is such a way that $m_gl_\delta-\bar{C}\theta>0$, 
\begin{eqnarray}
	\int_{\Bflat^\pm}\frac{\varphi^2dx}{(\varepsilon+|\nabla \tilde w|^{p-1})^\sigma}&\leq&\bar{D}.
\end{eqnarray}
A classical application of Fatou Lemma, gives us
\begin{eqnarray}
	\int_{\Bflat^\pm}\frac{\varphi^2dx}{|\nabla \tilde w|^{(p-1)\sigma}}\leq\bar{D}\quad\Rightarrow \quad\int_{B_\rho(x_0)}\frac{dx}{|\nabla \tilde w|^{(p-1)\sigma}}\leq\bar{D}.
\end{eqnarray}
This also implicitly implies that the critical set has zero Lebesgue measure.
\end{proof}
\begin{proof}[Proof of Theorem \ref{Reg1}]
Let us split the  $|D^2 \tilde{u}|^q$  as 
\[|D^2 u|^q=|\nabla u|^{\frac{(p-2-\beta)q}{2}}|D^2 u|^q\frac{1}{|\nabla u|^{\frac{(p-2-\beta)q}{2}}}.\]
Then by Theorem \ref{stimad2} we obtain 
\begin{eqnarray}\nonumber
&&\int_{\Omega\setminus Z_ {u}}|D^2u|^qdx\\\nonumber
&&= \int_{\Omega\setminus Z_{u}}|\nabla u|^{\frac{(p-2-\beta)q}{2}}|D^2 u|^q\frac{1}{|\nabla u|^{\frac{(p-2-\beta)q}{2}}}dx\\\nonumber
&&\leq \left(\int_{\Omega\setminus Z_{u}}|\nabla u|^{p-2-\beta}|D^2 u|^2dx\right)^{\frac q2} \left(\int_{\Omega\setminus Z_{u}}\frac{1}{|\nabla u|^{\frac{(p-2-\beta)q}{2-q}}}dx\right)^{\frac{2-q}{2}}\\\nonumber
&&\leq C \left(\int_{\Omega\setminus Z_{u}}\frac{1}{|\nabla u|^{\frac{(p-2-\beta)q}{2-q}}}dx\right)^{\frac{2-q}{2}}.
\end{eqnarray}
Therefore we can choose $\beta <1$ such that \[\frac{(p-2-\beta)q}{2-q}<p-1\]
if $q<(p-1)/(p-2)$ (observe that for $p\geq 3$ and $\beta=1$, $(p-2-\beta)q/{(2-q)} <p-1$ iff $q<(p-1)/(p-2)$). Setting 
$V_{\varepsilon, i}=G_{\varepsilon}(u_i)$, the proof ends following the one of Theorem \ref{Reg}.
\end{proof}

\vspace{1cm}
					
\begin{center}{\bf Acknowledgements}\end{center}  
L. Montoro, L. Muglia and B. Sciunzi are partially supported by PRIN project 2017JPCAPN (Italy): Qualitative and quantitative aspects of nonlinear PDEs, and L. Montoro by  Agencia Estatal de Investigaci\'on (Spain), project PDI2019-110712GB-100.

\					
\begin{center}
 {\sc Data availability statement}\
All data generated or analyzed during this study are included in this published article.
\end{center}

					\
					
\begin{center}
{\sc Conflict of interest statement}
\
The authors declare that they have no competing interest.
\end{center}

\bibliographystyle{elsarticle-harv}

				\end{document}